\documentclass{birkjour_t2}

\usepackage{graphicx}
\usepackage{mathptmx}
\usepackage{algorithm}
\usepackage{amsmath}
\usepackage{amssymb}
\usepackage{epstopdf}
\usepackage[dvips]{epsfig}
\usepackage{subfig}
\usepackage{url}
\usepackage[misc,geometry]{ifsym} 
\usepackage{hyperref}

\jourtitle{Mathematics in Computer Science}

\allowdisplaybreaks

\newtheorem{theorem}{Theorem}
\newtheorem{corollary}[theorem]{Corollary}
\newtheorem{lemma}[theorem]{Lemma}
\theoremstyle{definition}
\newtheorem{definition}[theorem]{Definition}

\begin{document}

\title[Global stability of a Caputo fractional SIRS model with general incidence rate]{%
Global stability of a Caputo fractional SIRS model\\ 
with general incidence rate}

\thanks{This is a preprint of a paper whose final and definite form is with 
\emph{Math. Comput. Sci.}, ISSN 1661-8270 (print), ISSN 1661-8289 (electronic), 
available at [\url{https://www.springer.com/journal/11786}]. 
Submitted 1-Jun-2019; Revised 5-Feb-2020; Accepted 6-Feb-2020.}


\author[M. R. Sidi Ammi \and M. Tahiri \and D. F. M. Torres]{%
Moulay Rchid Sidi Ammi$^1$ \and Mostafa Tahiri$^2$ \and Delfim F. M. Torres$^3$\\}

\address{$^1$Department of Mathematics, AMNEA Group, Faculty of Sciences and Techniques,\\
Moulay Ismail University, B.~P. 509, Errachidia, Morocco\\
\email{ORCID}{rachidsidiammi@yahoo.fr, \url{http://orcid.org/0000-0002-4488-9070}\\}}

\address{$^2$Department of Mathematics, AMNEA Group, Faculty of Sciences and Techniques,\\
Moulay Ismail University, B.~P. 509, Errachidia, Morocco\\
\email{}{my.mustafa.tahiri@gmail.com\\}}

\address[\Letter]{$^3$Center for Research and Development in Mathematics and Applications (CIDMA),\\
Department of Mathematics, University of Aveiro, 3810-193 Aveiro, Portugal\\
\email{ORCID}{delfim@ua.pt, \url{http://orcid.org/0000-0001-8641-2505}}}


\date{Received: 1 June 2019 / Revised: 5 February 2020 / Accepted: 6 February 2020}


\begin{abstract}
We introduce a fractional order SIRS model with non-linear incidence rate.
Existence of a unique positive solution to the model is proved.
Stability analysis of the disease free equilibrium
and positive fixed points are investigated.
Finally, a numerical example is presented.
\end{abstract}



\maketitle


\section{Introduction}
\label{sec1}

Fractional differential equations (FDEs) are generalizations
of classical differential equations, where the integer-order 
derivative is replaced by a non-integer one.
There has been a significant development in FDEs in recent years
due its applicability in different fields of science and engineering
\cite{MR3787674,MR3822307}. In particular, fractional derivatives 
are used to describe viscoelastic properties of many polymeric materials 
\cite{ref17}, in diffusion equations \cite{ref7}, in mechanics \cite{ref1}, 
and decision-making problems \cite{ref3}.

It is worthwhile to mention that fractional derivatives are non-local operators 
and thus may be more suitable for modelling systems dependent on past history (memory). 
More precisely, the fractional derivative of a given function does not depend 
only on its current state, but also on previous historical states \cite{ref11,MR1347689}.

In epidemiology, most mathematical models descend from the classical
SIR model of Kermack and McKendrick, established in 1927 \cite{ref12,MR3703345}.
Recently, fractional derivatives have been used to describe epidemiological models 
and, in some cases, they have proven to be more accurate when compared to the 
classical ones \cite{MR3770782,MR3743014,MyID:419}. Different models described 
by fractional derivatives are available in the literature, like the SIR model 
\cite{ref29,ref8,ref30}, the SIR model with vaccination \cite{ref18}, 
the SIRC model \cite{ref31}, and the SEIR model \cite{ref32}.

Since the fractional order can be any positive real $\alpha$, 
one can choose the one that better fits available data \cite{MR3719831}. 
Therefore, we can adjust the model to real data and, by doing so, better 
predict the future evolution of the disease taking into account its past
and present \cite{MyID:419,MR3789859}. 
Moreover, virus propagation is typically discontinuous, something 
the classical differential models cannot describe in a proper way.
In contrast, fractional systems deal naturally with such discontinuous 
properties \cite{MR3673702,MR3854267}. 

The virus propagation is similar to heat transmission 
or moistness penetrability in a porous medium, which 
can be exactly modelled by fractional calculus \cite{ref35,ref34}. 
The authors in \cite{ref36,ref37} give a geometrical description 
of fractional calculus, concluding that the fractional order
can be related with the fractal dimension.
The relationship between fractal dimension and fractional calculus 
has been obtained by several different authors: see \cite{MR3320677,ref38} 
and references therein. The fractional complex transform \cite{ref39,ref40} 
is an approximate transform of a fractal space (time) to a continuous one, 
and it is now widely used in fractional calculus \cite{MR3571716,ref41,ref42}. 

There are several definitions of fractional derivatives \cite{MR3822307,MR3331286}.
In this paper, we choose to work with the celebrated Caputo fractional derivatives. 
One of the main advantages of such derivatives is allowing us to consider classical 
initial conditions in the formulation of the problem. Also, 
the Caputo fractional derivatives of a constant are zero.
Such properties of the Caputo fractional derivatives are not true 
for other fractional operators, for example for the Riemann--Liouville 
derivatives \cite{ref11,MR1347689}.

Most non-linear fractional differential equations do not have analytic 
solutions \cite{MR3443073,MR3154620}. Therefore, approximations 
and numerical techniques must be used \cite{ref43,ref15,MR3854267}. 
The decomposition method \cite{ref44} and the variational 
iteration method \cite{ref46,ref45} are relatively new approaches 
to provide an analytical approximate solution to linear 
and non-linear problems. For a simple algorithm, based on
fractional Euler's method, to numerically solve non-linear fractional 
differential equations, in a direct way, without using linearisation, 
perturbations, or restrictive assumptions, see \cite{ref15}.

Here we propose a fractional SIRS model with the spread 
of the disease being described by a system of non-linear 
fractional order differential equations as follows:
\begin{equation}
\label{Eq1}
\left\{
\begin{array}{lll}
D^{\alpha}S(t) = \Lambda-\mu S(t)
-\dfrac{\beta S(t) I(t)}{1+k_1 S(t)+k_2 I(t)+k_3 S(t) I(t)}+\lambda R(t),\\[2ex]
D^{\alpha}I(t) = \dfrac{\beta S(t) I(t)}{1+k_1 S(t)+k_2 I(t)+k_3 S(t) I(t)}-(\mu+r)I(t),\\[2ex]
D^{\alpha}R(t) = r I(t)-(\mu+\lambda)R(t),
\end{array}
\right.
\end{equation}
where $D^{\alpha}$ denotes the (left) Caputo fractional derivative 
of order $\alpha$, $0 < \alpha \leq 1$. The model considers a population 
that is divided into three subgroups: susceptible $S(t)$, infective $I(t)$, 
and recovered $R(t)$ individuals at time $t$. The positive constants 
$\Lambda$, $\beta$, $\mu$, and $r$, are the recruitment rate of the population, 
the infection rate, the natural death rate, and the recovery rate 
of the infective individuals, respectively. The rate that recovered 
individuals lose immunity and return to the susceptible class is $\lambda$. 
While contacting with infected individuals, the susceptible become infected 
at the incidence rate $\beta SI/(1+k_1 S+k_2 I+k_3 SI)$,
with $k_1$, $k_2$, and $k_3$ non-negative constants \cite{ref28}.
This incidence function generalizes several types of incidence rates, for example,
the traditional bilinear incidence rate, the saturated incidence rate,
the Beddington--DeAngelis functional response proposed in \cite{ref2,ref5}, 
and the Crowley--Martin functional response introduced in \cite{ref4}.
For the advantages of using a general incidence rate, see \cite{MyID:427,MR3815138}.
For other ways to fractionalize a classical system of differential equations, 
see the discussion in \cite{MR3808497,MR3928263}.

The paper is organized as follows. In Section~\ref{sec2},
we recall necessary definitions and properties from
fractional calculus. Our results begin with Section~\ref{sec3},
where we show the existence and uniqueness of positive solution.
In Section~\ref{sec4}, we study the existence of equilibria and 
their local stability. The global stability is investigated 
in Section~\ref{sec5}. In order to illustrate our theoretical results,
numerical simulations of the model are given in Section~\ref{sec6}.
We end with Section~\ref{sec7} of conclusions and future perspectives. 

	
\section{Basic results of fractional calculus}
\label{sec2}

There are many good books on fractional calculus.
For a gentle introduction, we refer the reader to
\cite{ref11}. For an encyclopedic treaty, see \cite{MR1347689}.

\begin{definition}[See \cite{ref11}]
\label{def2.1}
The Riemann--Liouville fractional integral of order $\alpha > 0$
of a function $f : \mathbb{R}^+ \rightarrow \mathbb{R}$ is given by
$$
I^\alpha f(x) = \dfrac{1}{\Gamma(\alpha)}
\int_0^x (x-t)^{\alpha -1} f(t) dt,
$$
where $\Gamma(\alpha)=\int_0^\infty t^{\alpha-1} e^{-t} dt$
is the Euler Gamma function.
\end{definition}
	
\begin{definition}[See \cite{ref11}]
Let $\alpha> 0$, $n = [\alpha] + 1$, $n - 1 < \alpha \leq n$,
where $[\alpha]$ denotes the integer part of $\alpha$.
The Caputo fractional derivative of order $\alpha$ for a function
$f \in C^{n}([0, +\infty), \mathbb{R})$ is defined by
$$
D^\alpha f(u) = I^{n-\alpha} D^{n} f(u)
= \dfrac{1}{\Gamma(n-\alpha)} \int_0^u
\dfrac{f^{(n)} (s)}{(u-s)^{\alpha +1-n}} ds,
\quad u>0,
$$
where $D$ is the usual differential operator, that is, $D = \dfrac{d}{du}$.
In particular, when $0 < \alpha \leq 1$, one has
$$
D^\alpha f(u) = \dfrac{1}{\Gamma(1-\alpha)}
\int_0^u \dfrac{f' (s)}{(u-s)^{\alpha}} ds.
$$
\end{definition}

Next we recall the definition of the Mittag--Leffler function of parameter $\alpha$,
which is a generalization of the exponential function.

\begin{definition}[See \cite{ref11}]
\label{def 2.3}
Let $\alpha>0$. The function $E_\alpha$ defined by
$E_\alpha (z) = \displaystyle \sum_{j=0}^{\infty}\dfrac{z^j}{\Gamma(\alpha j + 1)}$
is called the Mittag--Leffler function of parameter $\alpha$.
\end{definition}

Let $f : \mathbb{R}^n \rightarrow \mathbb{R}^n$ with $n \geq 1$.
Consider the fractional order system
\begin{equation}
\label{eq2}
\left\{
\begin{array}{lll}
D^{\alpha}x(t) = f(x) ,\\
x(0) = x_0 ,
\end{array}
\right.
\end{equation}
where $0 < \alpha \leq 1$ and $x_0 \in \mathbb{R}^n$.
The following lemma, which is a direct corollary
from the main result of \cite{ref13},
gives global existence of solution to system~\eqref{eq2}.

\begin{lemma}[See \cite{ref13}]
\label{lem2.4}
Assume that $f$ satisfies the following conditions:
\begin{enumerate}
\item $f(x)$ and $\dfrac{\partial f}{\partial x}$
are continuous for all $x\in \mathbb{R}^n$;
\item $\|f(x)\|\leq\omega+\lambda \|x\|$  $\forall x \in \mathbb{R}^n$,
where $\omega$ and $\lambda$ are two positive constants.
\end{enumerate}
Then, system~\eqref{eq2} has a unique solution on $[0, +\infty)$.
\end{lemma}


\section{Existence and uniqueness of positive solution}
\label{sec3}

Denote $\mathbb{R}_+^3 = \{X \in \mathbb{R}^3 : X\geq0\}$ and let
$X(t) = (S(t), I(t), R(t))^T$.  Then system \eqref{Eq1} can be
reformulated as follows: $D^{\alpha}X(t) = F(X(t))$, where	
\begin{equation}
\label{eq3}
F(X)=
\left(
\begin{array}{ccc}
\Lambda-\mu S-\dfrac{\beta SI}{1+k_1 S+k_2 I+k_3 SI}+\lambda R\\[2ex]
\dfrac{\beta S I}{1+k_1 S+k_2 I+k_3 SI}-(\mu+r)I\\[2ex]
rI-(\mu+\lambda)R
\end{array}
\right).
\end{equation}
For biological reasons, we consider system~\eqref{Eq1}
with the following initial conditions:
\begin{equation}
\label{eq4}
S(0)\geq 0,\quad I(0)\geq0, \quad R(0)\geq 0.
\end{equation}
To prove the main theorem of this section, i.e., Theorem~\ref{thm3.3},
we need the following generalized mean value theorem and its corollary.

\begin{lemma}[Generalized Mean Value Theorem~\cite{ref20}]
\label{lem3.1}
Suppose that $f \in C[0, b]$ and $D^\alpha f
\in C(0, b]$, $0<\alpha \leq 1$.
Then, one has
$$
f(x) = f(0) + \dfrac{1}{\Gamma(\alpha)} (D^\alpha f)(\xi) x^\alpha
$$
with $0\leq \xi \leq x$ , $\forall x \in (0, b]$.
\end{lemma}

\begin{corollary}[See \cite{ref20}]
\label{cor3.2}
Suppose that $f \in C[0, b]$ and $D^\alpha f \in C(0, b]$
for $0<\alpha \leq 1$.  If $D^\alpha f(x)\geq 0$
$\forall x \in (0, b)$, then $f(x)$ is non-decreasing for each
$x\in[0, b]$. If $D^\alpha f(x)\leq 0$ $\forall\ x \in (0, b)$,
then $f(x)$ is  non-increasing for each $x\in[0, b]$.
\end{corollary}

\begin{theorem}
\label{thm3.3}
There is a unique solution for \eqref{Eq1} satisfying \eqref{eq4}
for $t \geq 0$ and the solution will remain in $\mathbb{R}_{+}^{3}$
for all $t \geq 0$. Moreover, $N(t) \leq N(0) + \dfrac{\Lambda}{\mu}$,
where $N(t) = S(t) + I(t) + R(t)$.
\end{theorem}

\begin{proof}
Since the vector function $F$ \eqref{eq3} satisfies
the first condition of Lemma~\ref{lem2.4}, we only need
to prove the second one. Denote
\begin{gather*}
\varepsilon=
\left(
\begin{array}{ccc}
\Lambda\\[2ex]
0\\[2ex]
0
\end{array}
\right),
\quad
A_1=
\left(
\begin{array}{ccc}
-\mu &    0      & \lambda\\[2ex]
0  & -(\mu+r)& 0\\[2ex]
0  &    r      &-(\mu+\lambda)
\end{array}
\right),
\quad
A_2=
\left(
\begin{array}{ccc}
0 & -\beta/k_1 & 0\\[2ex]
0 & \beta/k_1  & 0\\[2ex]
0 &    0       & 0
\end{array}
\right),\\
A_3=
\left(
\begin{array}{ccc}
-\beta/k_2 & 0 & 0\\[2ex]
\beta/k_2 & 0 & 0\\[2ex]
0 & 0 & 0
\end{array}
\right),
\quad
A_4=
\left(
\begin{array}{ccc}
-\beta/k_3 \\[2ex]
\beta/k_3 \\[2ex]
0
\end{array}
\right),
\quad
A_5=
\left(
\begin{array}{ccc}
-\beta & 0 & 0\\[2ex]
\beta & 0 & 0\\[2ex]
0 & 0 & 0
\end{array}
\right).
\end{gather*}	
We discuss four cases, as follows.\\
Case 1. If $k_1 \neq 0$, then we have
$$
F(X) = \varepsilon + A_1 X + \dfrac{k_1 S}{1+k_1 S+k_2 I+k_3 SI} A_2 X.
$$
Therefore,
$$
\|F(X)\| \leq \|\varepsilon\|+ \|A_1 X\| + \|A_2 X\|
= \|\varepsilon\|+ (\|A_1\| + \|A_2\|)\|X\|.
$$
Case 2. If $k_2 \neq 0$, then
$$
F(X) = \varepsilon + A_1 X + \dfrac{k_2 I}{1+k_1 S+k_2 I+k_3 SI} A_3 X.
$$
Thus,
$$
\|F(X)\| \leq \|\varepsilon\|+ \|A_1 X\| + \|A_3 X\|
= \|\varepsilon\|+ (\|A_1\| + \|A_3\|)\|X\|.
$$
Case 3. If $k_3 \neq 0$, then one has
$$
F(X) = \varepsilon + A_1 X + \dfrac{k_3 SI}{1+k_1 S+k_2 I+k_3 SI} A_4.
$$
We conclude that
$$
\|F(X)\| \leq \|\varepsilon\|+ \|A_4 \| + \|A_1\| \|X\|.
$$
Case 4. If $k_1 = k_2 = k_3 = 0$, then we obtain
$$
F(X) = \varepsilon + A_1 X + I A_5 X.
$$
It follows that
$$
\|F(X)\| \leq \|\varepsilon\|+ (\|A_1\| + \|I\|\|A_5\|)\|X\|.
$$
By Lemma~\ref{lem2.4}, it follows that \eqref{Eq1} subject to \eqref{eq4} 
has a unique solution. Now we prove the non-negativity of the solution. 
Observe first that
$$
D^{\alpha}S/_{S=0} = \Lambda + \lambda R,
$$
$$
D^{\alpha}I/_{I=0} = 0,
$$
and
$$
D^{\alpha}R/_{R=0} = rI.
$$
We can prove that the solution of \eqref{Eq1} remains non-negative 
for all $t \geq 0$ by proceeding in a similar way as in 
\cite[Theorem 2]{ref33}, that is, by considering an auxiliary system 
of fractional differential equations and by \emph{reductio ad absurdum}.
For that we use Corollary~\ref{cor3.2}, which is a consequence of Lemma~\ref{lem3.1}, 
to get a contradiction and then arriving to the intended conclusion by \cite[Lemma~1]{ref33}.
Finally, we establish the boundedness of solution. By summing
all the equations of system \eqref{Eq1}, we obtain that
\begin{equation*}
D^{\alpha}N = \Lambda-\mu N.
\end{equation*}
Solving this equality, we get
\begin{equation*}
N(t) \leq N(0) E_\alpha (-\mu t^\alpha)
+ \dfrac{\Lambda}{\mu}(1-E_\alpha (-\mu t^\alpha)).
\end{equation*}
Because $0\leq E_\alpha (-\mu t^\alpha) \leq 1$,
we have $N(t)\leq N(0) + \dfrac{\Lambda}{\mu}$.
This completes the proof.
\end{proof}


\section{Local stability}
\label{sec4}

In this section, we firstly discuss the existence of equilibria
for model \eqref{Eq1}. Let
$$
R_0 = \dfrac{\beta \Lambda}{(\mu + \Lambda k_1)(\mu + r)}.
$$
We prove that model \eqref{Eq1} has two possible equilibria.

\begin{theorem}
\label{thm 4.1}
\begin{itemize}
\item[(i)] There is always a disease-free equilibrium $E_0=(S_0, 0, 0)$,
where $S_0 = \dfrac{\Lambda}{\mu}$.
		
\item[(ii)] If $R_0 > 1$, then there exists a unique endemic equilibrium
$$
E^* = \left(S^*, \dfrac{(\mu +\lambda)(\Lambda-\mu S^*)}{c},
\dfrac{r(\Lambda-\mu S^*)}{c} \right),
$$
where
\begin{gather*}
c = a(\mu + \lambda)-\lambda r, \quad a = \mu + r,\\
S^* = \dfrac{k_3 a\Lambda (\mu+\lambda)+k_1 ac
- \beta c - k_2 a\mu (\mu+\lambda)+\sqrt{\Delta}}{2k_3 a\mu (\mu+\lambda)},\\
\Delta=\left(\beta c-k_1 ac-k_3 a\Lambda(\mu+\lambda)
+k_2 a\mu(\mu +\lambda)\right)^2
+4k_3 a\mu \left(ac+k_2 a\Lambda(\mu+\lambda)\right).
\end{gather*}
\end{itemize}
\end{theorem}

\begin{proof}
(i) By direct calculation, we have that $E_0$
is the unique steady state of system \eqref{Eq1}.\\
(ii) To find the other equilibrium, we solve the system
\begin{equation}
\label{eq7}
F(X)=0.
\end{equation}
Let
$$
f(S,I)=\dfrac{\beta S}{1+k_1 S +k_2 I +k_3 SI}.
$$
From system \eqref{eq7}, we obtain $I=\dfrac{(\mu +\lambda)(\Lambda-\mu S)}{c}$,
$R=\dfrac{r(\Lambda-\mu S)}{c}$, and
$f\left(S,\dfrac{(\mu +\lambda)(\Lambda-\mu S^*)}{c}\right)=a$.
Since $c=\mu^2 +\mu \lambda + r\mu>0$, we get
$I\geq0$ if $S\leq \dfrac{\Lambda}{\mu}$.
Now we consider function
$$
g(S)=f\left(S,\dfrac{(\mu +\lambda)(\Lambda-\mu S^*)}{c} - a\right)
$$
defined on the interval $\left[0, \dfrac{\Lambda}{\mu}\right]$.
One has
$$
\dfrac{\partial f}{\partial S}
=\dfrac{\beta(1+k_2 I)}{(1+k_1 S +k_2 I +k_3 SI)^2}>0
$$
and
$$
\dfrac{\partial f}{\partial I}
=\dfrac{-\beta S(k_2 +k_3 S)}{(1+k_1 S +k_2 I +k_3 SI)^2}<0.
$$
Then $g'(S)>0$, which implies that $g$ is strictly increasing
on $\left[0,  \dfrac{\Lambda}{\mu}\right]$. Hence, if $R_0>1$, 
the system admits a unique endemic equilibrium $E^* = (S^*, I^*, R^*)$ 
with $S^* \in \left(0,  \dfrac{\Lambda}{\mu}\right)$, $I^*>0$,
and $R^*>0$. This completes the proof.
\end{proof}

Next, we study the local stability of the disease-free equilibrium $E_0$
and the endemic equilibrium $E^*$. The Jacobian matrix
of system \eqref{Eq1} at any equilibrium
$\bar{E} = (\bar{S}, \bar{I}, \bar{R})$ is given by
\begin{equation}
\label{eq8}
J_{\bar{E}}
=\left(
\begin{array}{ccc}
-\mu-\dfrac{\beta \bar{I}(1+k_2 \bar{I})}{(1+k_1 \bar{S}+k_2 \bar{I}
+ k_3 \bar{S} \bar{I})^2}
& -\dfrac{\beta \bar{S}(1+k_1 \bar{S})}{(1+k_1 \bar{S}+k_2 \bar{I}
+ k_3 \bar{S} \bar{I})^2} & \lambda\\[2ex]
\dfrac{\beta \bar{I}(1+k_2 \bar{I})}{(1+k_1 \bar{S}+k_2
\bar{I}+ k_3 \bar{S} \bar{I})^2}
& \dfrac{\beta \bar{S}(1+k_1 \bar{S})}{(1+k_1
\bar{S}+k_2 \bar{I}+ k_3 \bar{S} \bar{I})^2}-a
& 0\\[2ex]
0 & r & -(\mu+\lambda)
\end{array}
\right).
\end{equation}
We recall that a sufficient condition
for the local stability of $\bar{E}$ is
\begin{equation}
\label{eq9}
|arg(\xi_i)| > \dfrac {\alpha \pi}{2}, \quad i=1, 2,
\end{equation}
where $\xi_i$ are the eigenvalues of $J_{\bar{E}}$ (see \cite{ref16}).
First, we establish the local stability of $E_0$.

\begin{theorem}
\label{thm4.2}
If $R_0 < 1$, then the disease-free equilibrium $E_0$
is locally asymptotically stable.
\end{theorem}

\begin{proof}
At $E_0$, \eqref{eq8} becomes
$$
J_{E_0}=
\left(
\begin{array}{ccc}
-\mu
&    -\dfrac{\beta \Lambda}{\mu +k_1 \Lambda} & \lambda\\[2ex]
0  & \dfrac{\beta \Lambda}{\mu +k_1 \Lambda}-a     & 0\\[2ex]
0  &                       r                     &-(\mu+\lambda)
\end{array}
\right).
$$
Hence, the eigenvalues of $J_{E_0}$ are $\xi_{1} = -\mu$, $\xi_{2} = a (R_{0} - 1)$,
and $\xi_3 = -(\mu + \lambda)$. Clearly, $\xi_2$ satisfies condition
\eqref{eq9} if $R_0 < 1$, since $\xi_1$ and $\xi_3$ are negative, proving the desired result.
\end{proof}

We now establish the local stability of $E^*$.

\begin{theorem}
\label{thm4.3}
If $R_0 > 1$, then the endemic equilibrium
$E^*$ is locally asymptotically stable.
\end{theorem}

\begin{proof}
At equilibrium $E^*$, the characteristic equation for the corresponding
linearised system of model \eqref{Eq1} is $\xi^3 +a_1 \xi^2 +a_2 \xi +a_3 =0$,
where
$$
a_1 = a+2\mu +\lambda + \dfrac{\beta I^* (1+k_2 I^*)
-\beta S^* (1+k_1 S^*)}{(1+k_1 S^* +k_2 I^* +k_3 S^* I^*)^2},
$$
$$
a_2 = \mu a + (\mu +\lambda)\mu + (\mu+\lambda)a
+ \dfrac{(a+\mu + \lambda) \beta I^* (1+k_2 I^*)
-(2\mu+\lambda) \beta S^* (1+k_1 S^*)}{(1+k_1 S^* +k_2 I^* +k_3 S^* I^*)^2},
$$
and
$$
a_3 = (\mu + \lambda)\mu a + \dfrac{c \beta I^* (1+k_2 I^*)
-(\mu + \lambda)\mu \beta S^* (1+k_1 S^*)}{(1+k_1 S^* +k_2 I^* +k_3 S^* I^*)^2}.
$$
Let $D(f)$ denote the discriminant of polynomial
$f(\xi) = \xi^3 +a_1 \xi^2 +a_2 \xi +a_3$. Then,
$$
D(f)=18a_1 a_2 a_3 +(a_1 a_2)^2 -4a_3 a_1 ^3 -4a_2 ^3 -27 a_3 ^2.
$$
Suppose that $E^*$ exists in $\mathbb{R}_+^3$.
Based on \cite{ref19}, we have the following
conclusions by using Routh--Hurwitz conditions:
\begin{itemize}
\item[(i)] if $a_1 >0$, $a_3>0$, and $a_1 a_2> a_3$,
then $E^*$ is locally asymptotically stable for all $\alpha \in (0, 1]$;

\item[(ii)] if $D(f) < 0$, $a_1 \geq 0$, $a_2 \geq 0$, $a_3 > 0$,
$a_1 a_2 < a_3$, and $\alpha < \dfrac{2}{3}$,
then $E^*$ is locally asymptotically stable;

\item[(iii)] if $D(f) < 0$, $a_1 < 0$, $a_2 < 0$,
and $\alpha > \dfrac{2}{3}$, then $E^*$ is unstable;

\item[(iv)] if $D(f) < 0$, $a_1 > 0$, $a_2 > 0$, $a_1 a_2 = a_3$,
and $\alpha \in (0, 1]$, then $E^*$ is locally asymptotically stable.

\item[(v)] if $D(f) < 0$, $a_1 > 0$, $a_3 = 0$, and $\alpha \in (0, 1]$,
then $E^*$ is locally stable.
\end{itemize}
The proof is complete.
\end{proof}


\section{Global stability}
\label{sec5}

In this section, we investigate the global stability
of both equilibria $E_0$ and ${E^*}$.

\begin{theorem}
\label{thm5.1}
The disease-free equilibrium $E_0$ is globally asymptotically
stable whenever $R_0 \leq 1$.
\end{theorem}

\begin{proof}
Consider the following Lyapunov function:
$$
L_0(t)=\dfrac{S_0}{1+k_1 S_0} \Psi\left(\dfrac{S(t)}{S_0}\right) + I(t)
+ \dfrac{1}{S_0 (1+k_1 S_0)}\left(\dfrac{\lambda}{r} \dfrac{R^2(t)}{2}
+ \dfrac{\lambda}{4\mu} \dfrac{(N(t)-S_0)^2}{2}\right),
$$
where $\Psi(x) = x-1-\ln(x)$, $x > 0$. Calculating the derivative
of $L_0$ along the solution of system \eqref{Eq1},
and by using Lemma~1 in \cite{ref14} and Lemma~3.1 in \cite{ref6},
we obtain that
\begin{align*}
D^\alpha L_0(t)
&\leq \dfrac{1}{1+k_1 S_0} \left(1-\dfrac{S_0}{S(t)}\right) D^\alpha S(t)
+ D^\alpha I(t) + \dfrac{\lambda}{r S_0 (1+k_1 S_0)}R(t) D^\alpha R(t)\\
&\quad + \dfrac{\lambda}{4 \mu S_0 (1+k_1 S_0)}(N(t)-S_0)D^\alpha N(t)\\
&\leq \dfrac{1}{1+k_1 S_0}\left(1-\dfrac{S_0}{S(t)}\right)\mu(S_0 - S(t))
- \dfrac{1}{1+k_1 S_0}\left(1-\dfrac{S_0}{S(t)}\right)
\dfrac{\beta S(t) I(t)}{1+k_1 S(t)+k_2 I(t)+k_3 S(t) I(t)}\\
&\quad + \dfrac{\lambda}{1+k_1 S_0}\left(1-\dfrac{S_0}{S(t)}\right) R(t)
+ \dfrac{\beta S(t) I(t)}{1+k_1 S(t)+k_2 I(t)+k_3 S(t) I(t)} - a I(t) \\
&\quad +\dfrac{\lambda}{r S_0 (1+k_1 S_0)}R(t) \left(r I(t) - (\mu + \lambda)R(t)\right)
+ \dfrac{\lambda}{4 \mu S_0 (1+k_1 S_0)}(N(t)-S_0)(\Lambda - \mu N(t))\\
&\leq \dfrac{-\mu}{(1+k_1 S_0)S(t)}(S(t)-S_0)^2 +a(R_0 - 1)I(t)
+ \dfrac{\lambda}{S_0 (1+k_1 S_0)}R(t) (N(t)-S_0) \\
&\quad + \dfrac{\lambda}{1+k_1 S_0}
R(t) \left(\dfrac{S_0-S(t)}{S_0} + \dfrac{S(t)-S_0}{S(t)}\right)
- \dfrac{\lambda}{S_0 (1+k_1 S_0) } R^2(t) \\
& \quad - \dfrac{\lambda}{S_0 (1+k_1 S_0) r}(\mu + \lambda)R^2(t)
- \dfrac{ \lambda}{4 S_0 (1+k_1 S_0)}(N(t)-S_0)^2\\
&\leq \dfrac{-\mu}{(1+k_1 S_0)S(t)}(S(t)-S_0)^2 +a(R_0 - 1)I(t)
- \dfrac{\lambda}{(1+k_1 S_0)}R(t) \dfrac{(S_0-S(t))^2}{S_0 S(t)} \\
&\quad - \dfrac{\lambda}{S_0 (1+k_1 S_0) } \left(R(t)-\dfrac{N(t)-S_0}{2}\right)^2
- \dfrac{\lambda}{S_0 (1+k_1 S_0) r}(\mu + \lambda)R^2(t).
\end{align*}
Therefore, $D^\alpha L_0(t) \leq 0$ if $R_0 \leq 1$. Furthermore, it is not
hard to verify that the largest compact invariant set of
$\left\{(S,I,R)\in \mathbb{R}_+^3 : D^\alpha L_0(t) = 0\right\}$
is the singleton $\{E_0\}$. Therefore, from LaSalle invariance principle \cite{ref10},
we deduce that ${E_0}$ is globally asymptotically stable.
\end{proof}

Finally, we investigate the global stability of the endemic equilibrium  ${E^*}$.

\begin{theorem}
\label{thm5.2}
The endemic equilibrium ${E^*}$ is globally asymptotically
stable if $R_0 > 1$ and
\begin{equation}
\label{eq10}
R^* \leq \dfrac{\mu}{\lambda} S^*.
\end{equation}
\end{theorem}

\begin{proof}
To study the global stability of ${E^*}$ for \eqref{Eq1},
we propose the following Lyapunov function:
\begin{equation*}
\begin{split}
L^*(t)&= \dfrac{1+k_2 S^*}{1+k_1 S^* + k_2 I^*
+ k_3 S^*I^*} S^{*^2} \Psi\left(\dfrac{S(t)}{S^*}\right)
+ S^{*^2} I^* \Psi\left(\dfrac{I(t)}{I^*}\right)\\
&\quad + \dfrac{\lambda (1+k_2 S^*)}{4\mu (1+k_1 S^* + k_2 I^* + k_3 S^* I^*)}
\Big[(S(t)-S^*)+(I(t)-I^*)+(R(t)-R^*)\Big]^2\\
&\quad + \dfrac{\lambda (1+k_2 S^*)}{2 r (1+k_1 S^* + k_2 I^* + k_3 S^*I^*)}(R(t)-R^*)^2,
\end{split}
\end{equation*}
where $\Psi(x) = x-1-\ln(x)$, $x > 0$. It follows that
\begin{equation*}
\begin{split}
D^\alpha L^*(t)
&\leq \dfrac{1+k_2 S^*}{1+k_1 S^* + k_2 I^* + k_3 S^*I^*}
S^* \left(1-\dfrac{S^*}{S(t)}\right) D^\alpha S(t)
+ S^* \left(1-\dfrac{I^*}{I(t)}\right) D^\alpha I(t)\\
&\quad + \dfrac{\lambda (1+k_2 S^*)}{2 \mu (1+k_1 S^* + k_2 I^* + k_3 S^*I^*)}
\Big[(S(t)-S^*)+(I(t)-I^*)+(R(t)-R^*)\Big] D^\alpha N(t) \\
&\quad + \dfrac{\lambda (1+k_2 S^*)}{r
(1+k_1 S^* + k_2 I^* + k_3 S^*I^*)}(R(t)-R^*)D^\alpha R(t).
\end{split}
\end{equation*}
Note that
$\Lambda=\mu S^* + a I^* - \lambda R^* = \mu (S^* +I^* +R^*)$,
$$
\dfrac{\beta S^*}{1+k_1 S^* + k_2 I^* + k_3 S^*I^*}=a,
$$
and $(\lambda + \mu)R^* - rI^* =0$. Then,
\begin{align*}
D^\alpha L^*(t)
&\leq \dfrac{1+k_2 S^*}{1+k_1 S^* + k_2 I^* + k_3 S^*I^*}
S^*\left(1-\dfrac{S^*}{S(t)}\right)\Bigg(\Lambda
-\mu S(t)-\dfrac{\beta S(t)I(t)}{1+k_1 S(t)+k_2 I(t)+k_3 S(t)I(t)}\\
&\quad +\lambda R(t)\Bigg) + S^* \left(1-\dfrac{I^*}{I(t)}\right)\left(
\dfrac{\beta S(t) I(t)}{1+k_1 S(t)+k_2 I(t)+k_3 S(t)I(t)}-aI(t)\right) \\
&\quad + \dfrac{\lambda (1+k_2 S^*)}{2 \mu (1+k_1 S^* + k_2 I^* + k_3 S^*I^*)}
\Big[(S(t)-S^*)+(I(t)-I^*)+(R(t)-R^*)\Big]\Big(\Lambda-\mu N(t)\Big) \\
&\quad + \dfrac{\lambda (1+k_2 S^*)}{r
(1+k_1 S^* + k_2 I^* + k_3 S^*I^*)}\left(R(t)-R^*\right)
\left(rI(t)-(\mu+\lambda)R(t)\right)\\
&\leq \dfrac{\mu S^*(1+k_2 S^*)}{1+k_1 S^* + k_2 I^* + k_3 S^*I^*}
\dfrac{S(t)-S^*}{S(t)} (S^*-S(t)) + a I^* \dfrac{S^*(1+k_2 S^*)}{1
+k_1 S^* + k_2 I^* + k_3 S^*I^*} \dfrac{S(t)-S^*}{S(t)}\\
&\quad - \dfrac{S^*(1+k_2 S^*)}{1+k_1 S^*
+ k_2 I^* + k_3 S^*I^*} \dfrac{S(t)-S^*}{S(t)}
\dfrac{\beta S(t)I(t)}{1+k_1 S(t)+k_2 I(t)+k_3 S(t)I(t)}\\
&\quad + \dfrac{\lambda S^*(1+k_2 S^*)}{1
+k_1 S^* + k_2 I^* + k_3 S^*I^*}
\dfrac{S(t)-S^*}{S(t)}(R(t)-R^*)\\
&\quad + S^* \left(1-\dfrac{I^*}{I(t)}\right) 
\dfrac{\beta S(t) I(t)}{1+k_1 S(t)+k_2 I(t)+k_3 S(t)I(t)}-aS^*(I(t)-I^*)\\
&\quad + \dfrac{\lambda (1+k_2 S^*)}{2 \mu (1+k_1 S^* + k_2 I^*
+ k_3 S^*I^*)}\Big[(S(t)-S^*)+(I(t)-I^*)\\
&\quad + (R(t)-R^*)\Big]\Big[\mu(S^*-S(t))+\mu(I^*-I(t))+\mu(R^*-R(t))\Big]\\
&\quad + \dfrac{\lambda (1+k_2 S^*)}{r (1+k_1 S^*
+ k_2 I^* + k_3 S^*I^*)}\left(R(t)-R^*\right)
\left(r(I(t)-I^*)-(\mu+\lambda)(R(t)-R^*)\right)\\
&\leq -\dfrac{\mu S^*(1+k_2 S^*)}{1+k_1 S^* + k_2 I^* + k_3 S^*I^*}
\dfrac{(S(t)-S^*)^2}{S(t)} + a I^*\dfrac{S^*(1+k_2 S^*)}{1
+k_1 S^* + k_2 I^* + k_3 S^*I^*} \dfrac{S(t)-S^*}{S(t)}\\
&\quad - \dfrac{a (1+k_2 S^*)(S(t)-S^*)I(t)}{1+k_1 S(t)
+k_2 I(t) + k_3 S(t)I(t)} 
+ \left(1-\dfrac{I^*}{I(t)}\right)\dfrac{a(1+k_1 S^* + k_2 I^*
+ k_3 S^*I^*)S(t)I(t)}{1+k_1 S(t)+k_2 I(t) + k_3 S(t)I(t)}\\
&\quad  -aS^*I(t) + aS^*I^* + \dfrac{\lambda S^*(1+k_2 S^*)}{1
+k_1 S^* + k_2 I^* + k_3 S^*I^*} \dfrac{S(t)-S^*}{S(t)}(R(t)-R^*)\\
&\quad - \dfrac{\lambda (1+k_2 S^*)}{1
+k_1 S^* + k_2 I^* + k_3 S^*I^*} (S(t)-S^*)(R(t)-R^*)\\
&\leq -\dfrac{\mu S^*(1+k_2 S^*)}{1+k_1 S^* + k_2 I^* + k_3 S^*I^*}
\dfrac{(S(t)-S^*)^2}{S(t)} + a I^*\dfrac{S^*(1+k_2 S^*)}{1
+k_1 S^* + k_2 I^* + k_3 S^*I^*} \dfrac{S(t)-S^*}{S(t)}\\
&\quad +\dfrac{aS^*(1+k_1 S(t) + k_2 I^* + k_3 S(t)I^*)I(t)}{1
+k_1 S(t)+k_2 I(t) + k_3 S(t)I(t)} - \dfrac{aI^*(1+k_1 S^* 
+ k_2 I^* + k_3 S^*I^*)S(t)}{1+k_1 S(t)+k_2 I(t) + k_3 S(t)I(t)} -aS^*I(t)\\
&\quad  + aS^*I^* +\dfrac{\lambda (1+k_2 S^*)}{1+k_1 S^*
+ k_2 I^* + k_3 S^*I^*} (S(t)-S^*)(R(t)-R^*)\dfrac{S^*-S(t)}{S(t)}\\
&\leq -\dfrac{\mu S^*(1+k_2 S^*)}{1+k_1 S^* + k_2 I^* + k_3 S^*I^*}
\dfrac{(S(t)-S^*)^2}{S(t)} +  a I^* \dfrac{S^*(1+k_2 S^*)}{1+k_1 S^*
+ k_2 I^* + k_3 S^*I^*} \dfrac{S(t)-S^*}{S(t)}\\
&\quad + \dfrac{aS^*(1+k_1 S(t) + k_2 I^* 
+ k_3 S(t)I^*)I(t)}{1+k_1 S(t)+k_2 I(t) + k_3 S(t)I(t)}
- \dfrac{aI^*(1+k_1 S^* + k_2 I^* + k_3 S^*I^*)S(t)}{1+k_1 S(t)+k_2 I(t) 
+ k_3 S(t)I(t)} -aS^*I(t)\\
&\quad + aS^*I^* - \dfrac{\lambda (1+k_2 S^*)}{1+k_1 S^* + k_2 I^* + k_3 S^*I^*}
\dfrac{(S(t)-S^*)^2}{S(t)}(R(t)-R^*))\\
&\leq -\dfrac{\mu S^*(1+k_2 S^*)}{1+k_1 S^* + k_2 I^* + k_3 S^*I^*} \dfrac{(S(t)-S^*)^2}{S(t)}\\
&\quad + aS^*I^*\left(\dfrac{(1+k_2 S^*)}{1+k_1 S^* + k_2 I^* + k_3 S^*I^*} \dfrac{S(t)-S^*}{S(t)}
+ \dfrac{(1+k_1 S(t) + k_2 I^* + k_3 S(t)I^*)I(t)}{(1+k_1 S(t)+k_2 I(t) + k_3 S(t)I(t))I^*}\right)\\
&\quad + aS^*I^* \left(1-\frac{I(t)}{I^*} - \dfrac{(1+k_1 S^* + k_2 I^*
+ k_3 S^*I^*)S(t)}{(1+k_1 S(t)+k_2 I(t) + k_3 S(t)I(t))S^*}\right)\\
&\quad+ R^* \dfrac{\lambda (1+k_2 S^*)}{1+k_1 S^* + k_2 I^* + k_3 S^*I^*}
\dfrac{(S(t)-S^*)^2}{S(t)}\\
&\leq R^* \dfrac{\lambda (1+k_2 S^*)}{1+k_1 S^* + k_2 I^* + k_3 S^*I^*}
\dfrac{(S(t)-S^*)^2}{S(t)}-\dfrac{\mu S^*(1+k_2 S^*)}{1+k_1 S^*
+ k_2 I^* + k_3 S^*I^*} \dfrac{(S(t)-S^*)^2}{S(t)} \\
&\quad + aS^*I^*\Bigg(\dfrac{(1+k_2 S^*)}{1+k_1 S^* + k_2 I^* + k_3 S^*I^*}
\dfrac{S(t)-S^*}{S(t)} -1  +1-\dfrac{1+k_1 S(t)+k_2 I(t) 
+ k_3 S(t)I(t)}{1+k_1 S(t) + k_2 I^* + k_3 S(t)I^*} \\
&\quad - \dfrac{(1+k_1 S^* + k_2 I^* + k_3 S^*I^*)
S(t)}{(1+k_1 S(t)+k_2 I(t) + k_3 S(t)I(t))S^*}\Bigg)
+ aS^*I^*\Bigg(2-1-\frac{I(t)}{I^*}\\
&\quad + \dfrac{1+k_1 S(t)+k_2 I(t) 
+ k_3 S(t)I(t)}{1+k_1 S(t) + k_2 I^* + k_3 S(t)I^*} 
+ \dfrac{(1+k_1 S(t) + k_2 I^* + k_3 S(t)I^*)I(t)}{\left(1
+k_1 S(t)+k_2 I(t) + k_3 S(t)I(t)\right)I^*}\Bigg)\\
&\leq (\lambda R^* - \mu S^*) \dfrac{(1+k_2 S^*)}{1+k_1 S^* 
+ k_2 I^* + k_3 S^*I^*} \dfrac{(S(t)-S^*)^2}{S(t)}\\
&\quad + aS^*I^*\Bigg(3 - \dfrac{(1+k_1 S(t) + k_2 I^* 
+ k_3 S(t)I^*)S^*}{(1+k_1 S^* + k_2 I^* + k_3 S^*I^*)S(t)}
-\dfrac{1+k_1 S(t)+k_2 I(t) + k_3 S(t)I(t)}{1+k_1 S(t) 
+ k_2 I^* + k_3 S(t)I^*} \\
&\quad - \dfrac{(1+k_1 S^* + k_2 I^* + k_3 S^*I^*)S(t)}{\left(1
+k_1 S(t)+k_2 I(t) + k_3 S(t)I(t)\right)S^*}\Bigg)
+ aS^*I^*\Bigg(-1-\frac{I(t)}{I^*}\\
&\quad + \dfrac{1+k_1 S(t)+k_2 I(t) + k_3 S(t)I(t)}{1
+k_1 S(t) + k_2 I^* + k_3 S(t)I^*}
+ \dfrac{(1+k_1 S(t) + k_2 I^* + k_3 S(t)I^*)I(t)}{\left(1
+k_1 S(t)+k_2 I(t) + k_3 S(t)I(t)\right)I^*}\Bigg)\\
&\leq \left(\lambda R^* - \mu S^*\right) 
\dfrac{\left(1+k_2 S^*\right)}{1+k_1 S^* + k_2 I^*
+ k_3 S^*I^*} \dfrac{(S(t)-S^*)^2}{S(t)}\\
&\quad + aS^*I^*\Bigg[\Bigg(3 
- \dfrac{\left(1+k_1 S(t) + k_2 I^* + k_3 S(t)I^*\right)S^*}{\left(1+k_1 S^*
+ k_2 I^* + k_3 S^*I^*\right)S(t)}-\dfrac{1+k_1 S(t)+k_2 I(t) 
+ k_3 S(t)I(t)}{1+k_1 S(t) + k_2 I^* + k_3 S(t)I^*}\\
&\quad - \dfrac{(1+k_1 S^* + k_2 I^* + k_3 S^*I^*)S(t)}{\left(1
+k_1 S(t)+k_2 I(t) + k_3 S(t)I(t)\right)S^*}\Bigg)
+ \ln \dfrac{\left(1+k_1 S(t) + k_2 I^* + k_3 S(t)I^*\right)S^*}{\left(1
+k_1 S^* + k_2 I^* + k_3 S^*I^*\right)S(t)}\\
&\quad + \ln \dfrac{\left(1+k_1 S(t)+k_2 I(t) + k_3 S(t)I(t)\right)S^*}{\left(1
+k_1 S(t) + k_2 I^* + k_3 S(t)I^*\right)S^*}
+ \ln \dfrac{\left(1+k_1 S^* + k_2 I^* + k_3 S^*I^*\right)S(t)}{\left(1
+k_1 S(t)+k_2 I(t) + k_3 S(t)I(t)\right)S^*}\Bigg]\\
&\quad + aS^*I^*\left[\dfrac{1+k_1 S(t)+k_2 I(t) 
+ k_3 S(t)I(t)}{1+k_1 S(t) + k_2 I^* + k_3 S(t)I^*}-1\right]\\
&\quad + aS^*I(t)\left[\dfrac{1+k_1 S(t) + k_2 I^* + k_3 S(t)I^*}{1
+k_1 S(t)+k_2 I(t) + k_3 S(t)I(t)} - 1\right]\\
&\leq (\lambda R^* - \mu S^*) \dfrac{(1+k_2 S^*)}{1+k_1 S^* 
+ k_2 I^* + k_3 S^*I^*} \dfrac{(S(t)-S^*)^2}{S(t)}\\
&\quad - aS^*I^*\Bigg[\Psi\left(\dfrac{\left(1
+k_1 S(t) + k_2 I^* + k_3 S(t)I^*\right)S^*}{\left(1
+k_1 S^* + k_2 I^* + k_3 S^*I^*\right)S(t)}\right) 
+ \Psi\left(\dfrac{\left(1+k_1 S(t)+k_2 I(t) + k_3 S(t)I(t)\right)S^*}{\left(1
+k_1 S(t) + k_2 I^* + k_3 S(t)I^*\right)S^*}\right)\\
&\quad + \Psi\left(\dfrac{\left(1+k_1 S^* + k_2 I^* + k_3 S^*I^*\right)S(t)}{\left(1
+k_1 S(t)+k_2 I(t) + k_3 S(t)I(t)\right)S^*}\right)\Bigg]\\
&\quad- \dfrac{aS^*\left(k_2 + k_3 S(t)\right)\left(1+k_1 S(t)\right)
\left(I(t)-I^*\right)^2}{\left(1+k_1 S(t)+k_2 I(t) 
+ k_3 S(t)I(t)\right)\left(1+k_1 S(t) + k_2 I^* + k_3 S(t)I^*\right)}.
\end{align*}
Clearly, $\Psi(x) \geq 0$. Consequently, $D^{\alpha}L^*(t) \leq 0$
if $R_0 > 0$ and $R^* \leq \frac{\mu}{\lambda} S^*$.
Further, the largest invariant set of
$\{(S,I,R)\in \mathbb{R}_+^3 : D^\alpha L^*(t) = 0\}$
is the singleton $\{E^*\}$. Therefore, from LaSalle's
invariance principle, ${E^*}$ is globally asymptotically stable.
\end{proof}

It is easy to see that condition \eqref{eq10} 
in Theorem~\ref{thm5.2} is equivalent to
\begin{gather}
\label{eq10:quivA}
\begin{aligned}
R_0 \leq 1
&+\dfrac{k_3a\beta\Lambda^2(\lambda^2r+a(\mu+r))(\mu+r)(\mu+\lambda)
+k_1a^2\beta\lambda^2r(\mu+\lambda)^2}{\lambda^2k_3a(\mu+r)^2(\mu+\lambda)(\mu+k_1\Lambda)}\\
&+\dfrac{k_1a\beta(a^2(\mu+\lambda)^2+\lambda^2r^2)(\mu+\lambda)
+\beta(\mu+r)\sqrt{\Delta}}{\lambda^2k_3a(\mu+r)^2(\mu+\lambda)(\mu+k_1\Lambda)}
\end{aligned}
\end{gather}
and
\begin{gather}
\label{eq10:quivB}
\begin{aligned}
&\lim\limits_{\lambda \rightarrow 0}
\Bigg(\dfrac{k_3a\beta\Lambda^2(\lambda^2r+a(\mu+r))(\mu+r)(\mu+\lambda)
+k_1a^2\beta\lambda^2r(\mu+\lambda)^2}{\lambda^2k_3a(\mu+r)^2(\mu+\lambda)(\mu+k_1\Lambda)}\\
&\qquad\quad+\dfrac{k_1a\beta(a^2(\mu+\lambda)^2+\lambda^2r^2)(\mu+\lambda)
+\beta(\mu+r)\sqrt{\Delta}}{\lambda^2k_3a(\mu+r)^2(\mu+\lambda)(\mu+k_1\Lambda)}\Bigg)
=+\infty.
\end{aligned}
\end{gather}

\begin{corollary}
The endemic equilibrium ${E^*}$ is globally asymptotically
stable if $R_0 > 1$ and $\lambda$ is sufficiently small.
\end{corollary}


\section{Numerical simulations}
\label{sec6}

In this section we present some numerical simulations
in order to illustrate our theoretical results.
The system \eqref{Eq1} is numerically integrated by using
the fractional Euler's method \cite{ref15}.
The approximate solutions of \eqref{eq2} with $0 <\alpha \leq 1$
are displayed in Figures~\ref{fig01} and \ref{fig02}.
The solutions converge to the equilibrium points.
The parameter values used in the simulations are: $\Lambda = 0.8$, $\mu = 0.1$,
$\lambda=0.5$, $\beta = 0.1$, $r = 0.5$, $k_1 = 0.1$, $k_2 = 0.02$, and $k_3 = 0.003$
with initial conditions $S(0) = 10.0$, $I(0) = 1.0$, $R(0) = 1.0$. Using the
MATLAB numerical computing environment, we get $R_0 = 0.7407$. Hence,
system \eqref{Eq1} has a unique disease-free equilibrium
$E_0 = (8, 0, 0)$. According to Theorem~\ref{thm5.1}, 
$E_0$ is globally asymptotically stable (see Figure~\ref{fig01}).
\begin{figure}[ht]
\centering
\includegraphics[scale=0.50]{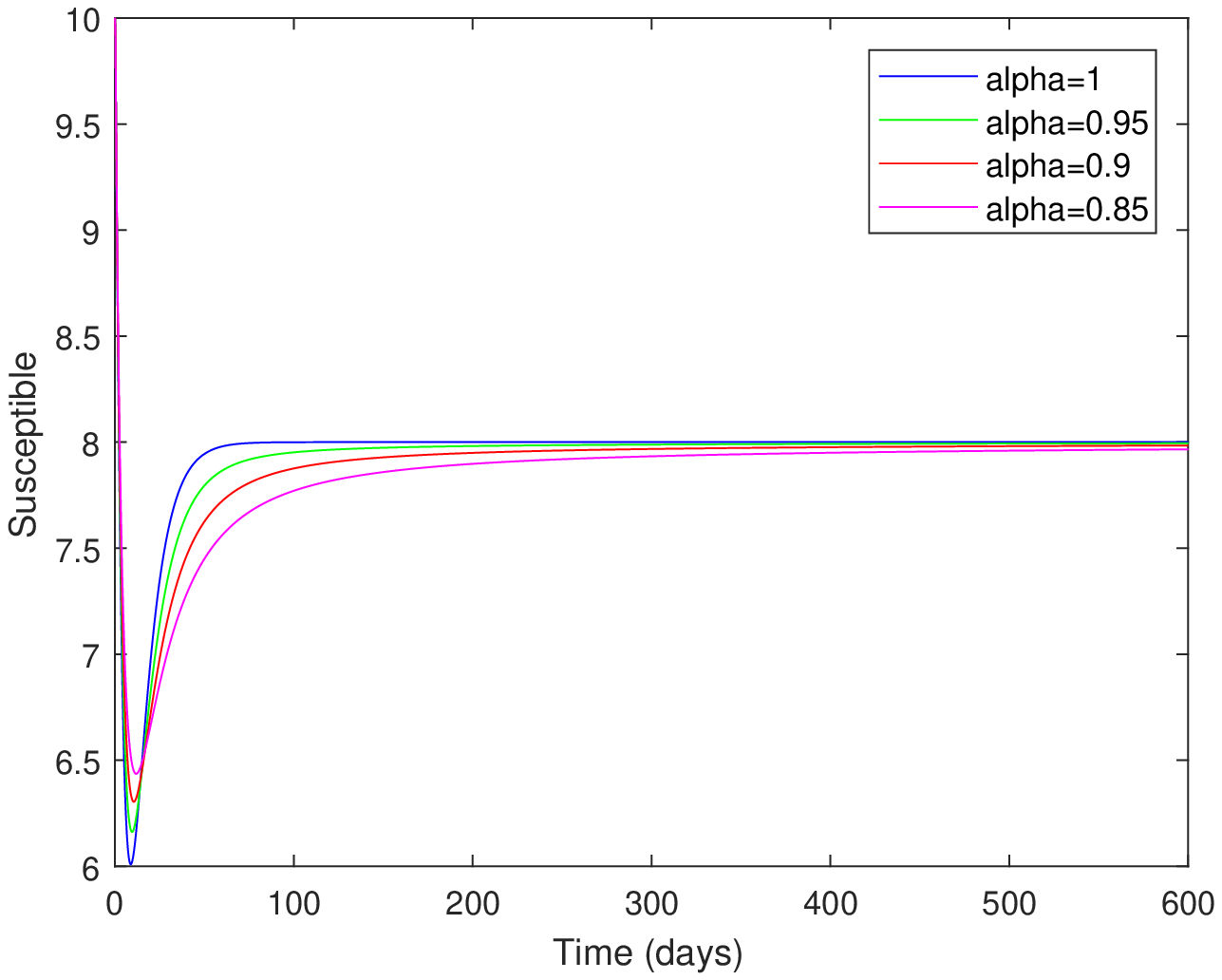}
\centering
\includegraphics[scale=0.50]{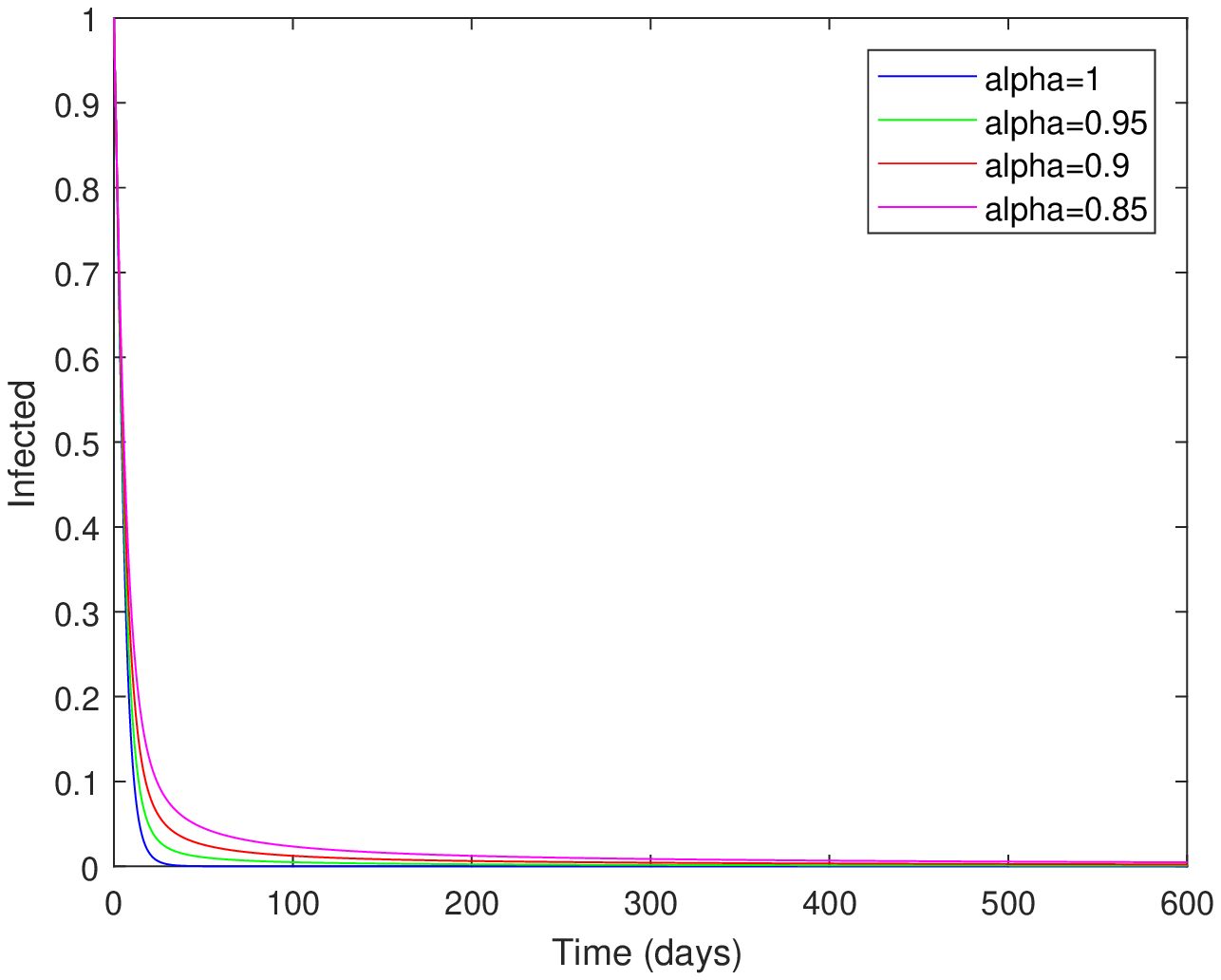}
\centering
\includegraphics[scale=0.50]{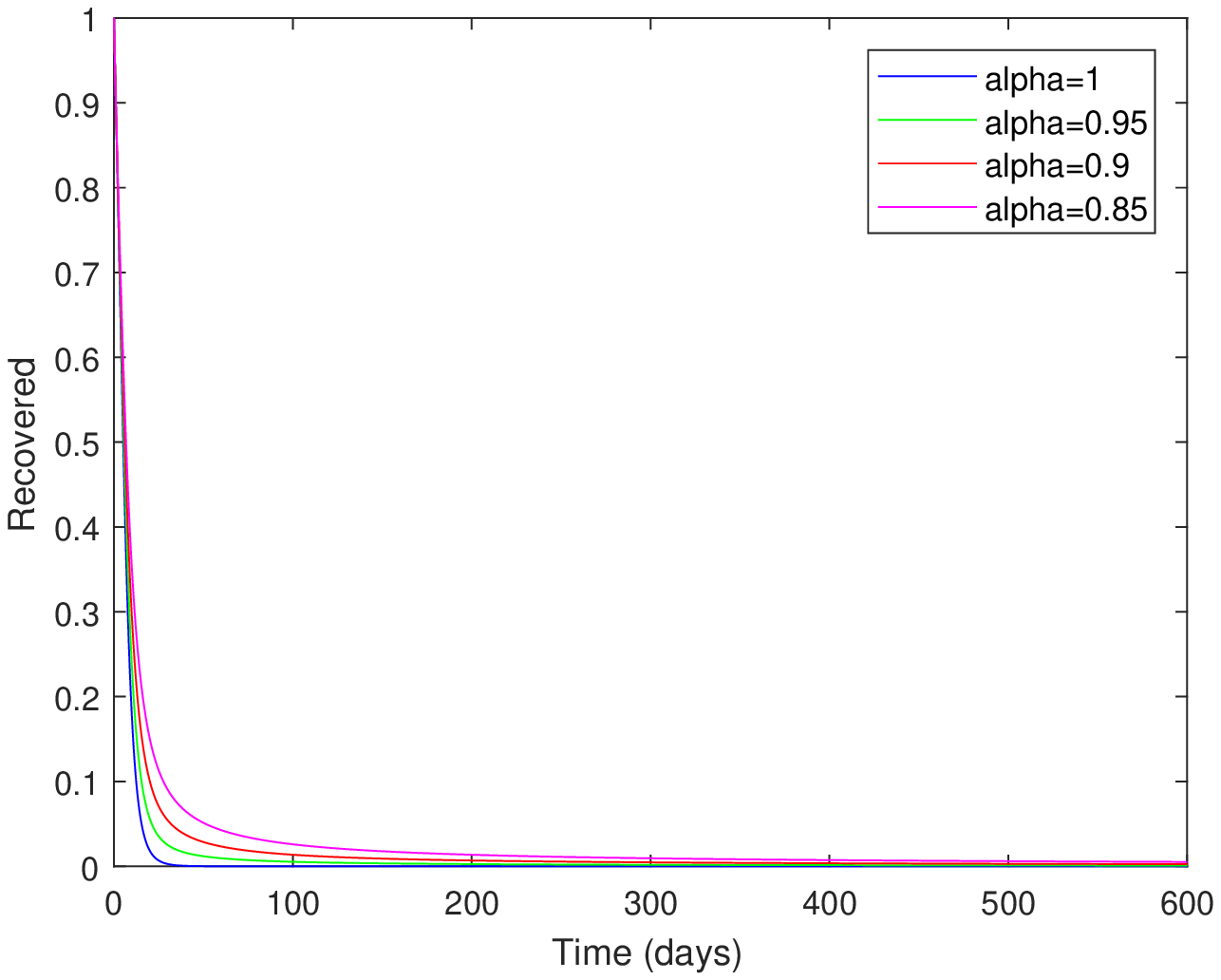}
\caption{Stability of the disease-free equilibrium $E_0$.}
\label{fig01}
\end{figure}

Now, let us choose $\mu = 0.02$ and keep the other parameter values.
Then, $R_0 = 1.5385$ and $R^* =0.552 \leq 21.8944$. Hence, the
condition~\eqref{eq10} is satisfied, as well as conditions
\eqref{eq10:quivA} and \eqref{eq10:quivB}, and the model converges
rapidly to its steady state for $\alpha =0.85$, when
compared to other fractional derivatives (see Figure~\ref{fig02}).
\begin{figure}[ht]
\centering
\includegraphics[scale=0.50]{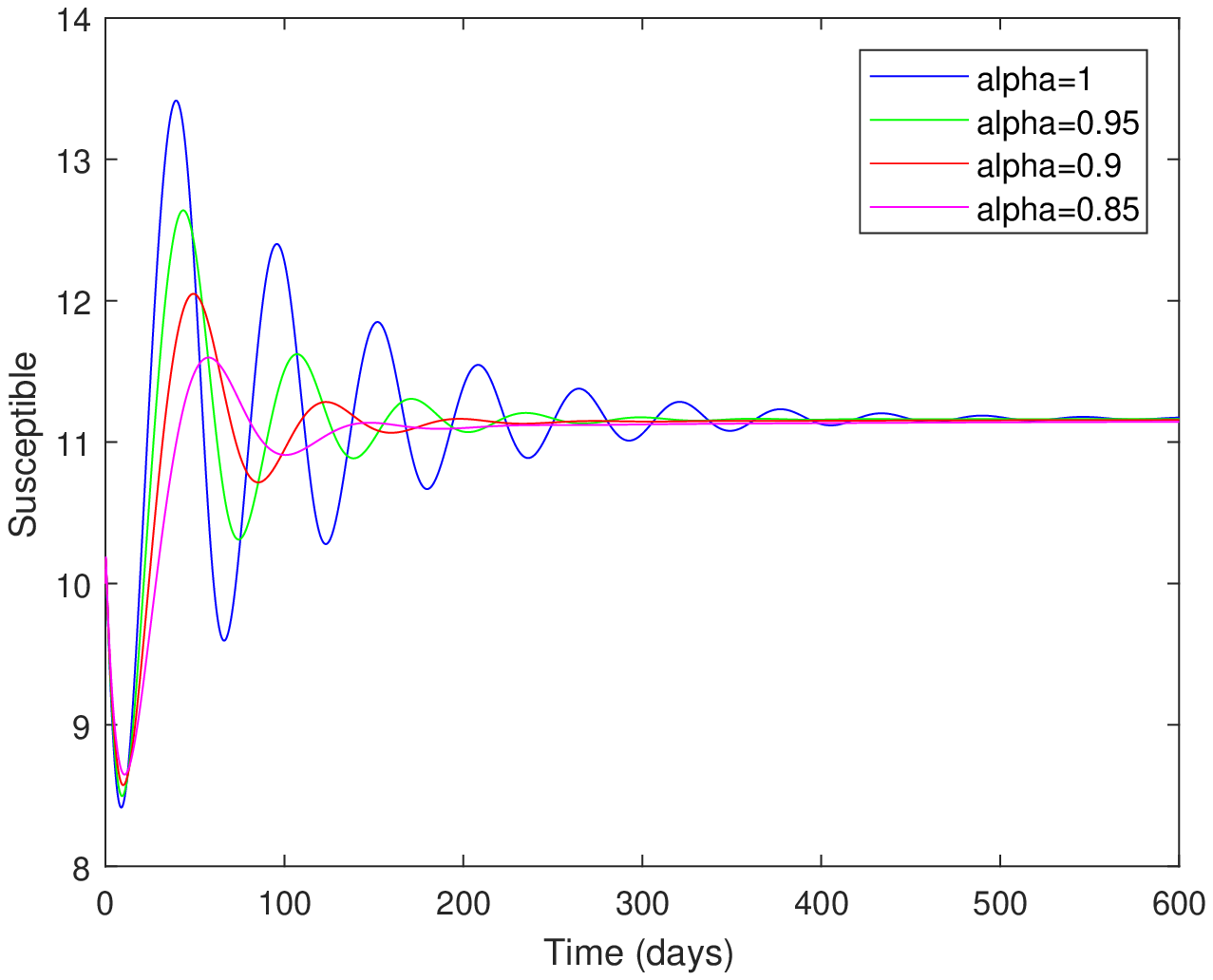}
\centering
\includegraphics[scale=0.50]{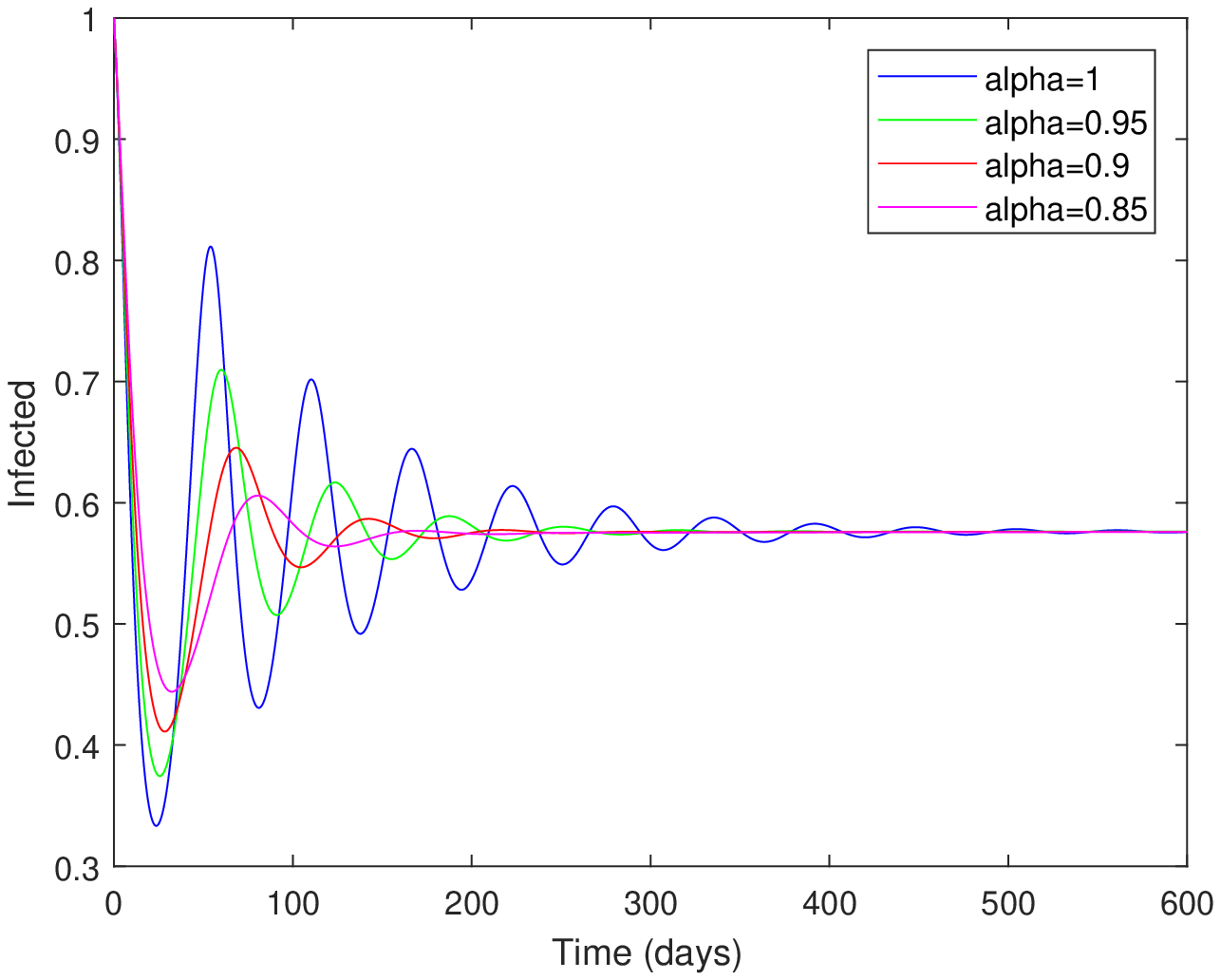}
\centering
\includegraphics[scale=0.50]{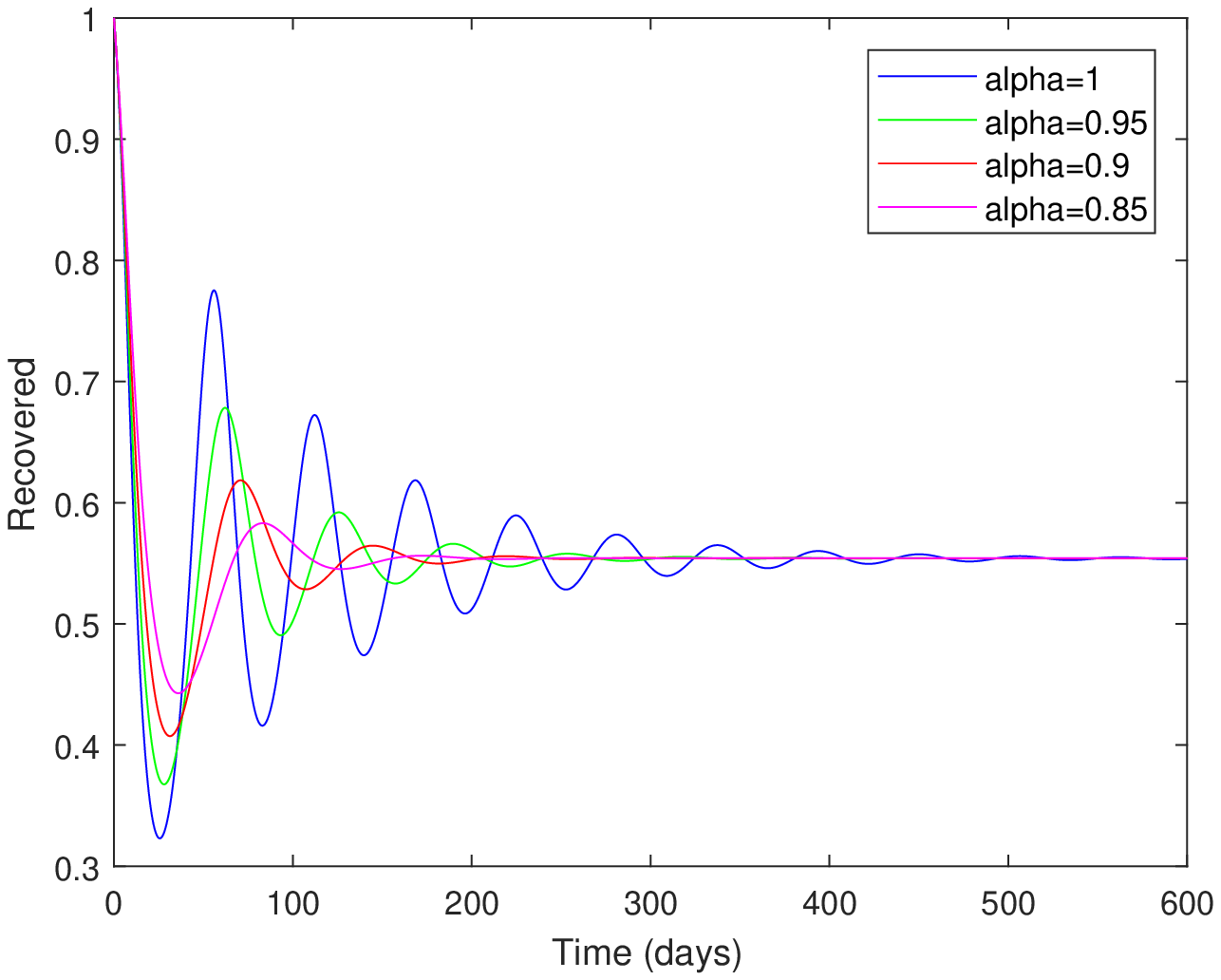}
\caption{Stability of the endemic equilibrium $E^*$.}
\label{fig02}
\end{figure}


\section{Conclusion}
\label{sec7}

The use of fractional order derivatives
can help to reduce errors arising from
the neglected parameters in modelling
real life phenomena \cite{ref8}.
Here, we have studied a fractional-order SIRS epidemic model
with a general incidence function.
The stability of equilibrium points
is investigated and numerical solutions are given.
According to our theoretical analysis, the fractional
order parameter $\alpha$ has no effect on the stability of free
and endemic equilibria, but it can affect the time
for arriving at the steady states. As future work, 
we plan to study the stability of a more general SIRS 
type model taking into account other parameters.


\section*{Acknowledgements}

This work was partially supported by
\emph{Funda\c{c}\~ao para a Ci\^encia e a Tecnologia} (FCT)
within the R\&D unit \emph{Centro de Investiga\c{c}\~{a}o 
e Desenvolvimento em Matem\'{a}tica e Aplica\c{c}\~{o}es} (CIDMA),
project UIDB/04106/2020. The authors are very grateful 
to two anonymous reviewers, for several critical remarks 
and precious suggestions.



\end{document}